\documentclass[12pt]{amsart}

\usepackage{enumitem}
\usepackage{psfrag}
\usepackage{graphicx}
\usepackage{pinlabel}
\usepackage{graphicx}
\usepackage{amsmath}
\usepackage{amssymb}
\usepackage{amscd}
\usepackage{times}
\usepackage{pb-diagram}
\usepackage{graphicx}
\usepackage{wrapfig}
\usepackage{xspace}
\usepackage{hyperref}
\usepackage{url}
\usepackage{caption}
\usepackage{subcaption}
\usepackage{fancyhdr}
\usepackage{overpic}
\usepackage{color}
\usepackage[square,comma,sort&compress,numbers]{natbib}
\usepackage{latexsym}
\usepackage{mathrsfs}
\usepackage{amsfonts}
\usepackage{hyperref}
\usepackage{amsthm}
\usepackage{appendix}
\usepackage{pdfsync}

\theoremstyle{definition}
\newtheorem{theorem}{Theorem}[section]
\newtheorem{lemma}[theorem]{Lemma}
\newtheorem{proposition}[theorem]{Proposition}

\newtheorem{definition}[theorem]{Definition}

\newtheorem{remark}[theorem]{Remark}
\newtheorem*{theorem*}{Theorem}
\begin{document}

\title{\bf 	Combinatorial Yamabe flow on infinitely triangulated hyperbolic surfaces  }
\author{YueRong Bian,
	Xu Xu}

\date{\today}

\address{School of Mathematics and Statistics, Wuhan University, Wuhan, 430072, P.R.China}
\email{yuerongbian@whu.edu.cn}

\address{School of Mathematics and Statistics, Wuhan University, Wuhan, 430072, P.R.China}
\email{xuxu2@whu.edu.cn}



\begin{abstract}
We study the combinatorial Yamabe flow on infinitely triangulated surfaces with piecewise hyperbolic metrics.
Under the assumptions of uniformly bounded vertex degree and
an $\epsilon$-uniformly nondegenerate initial metric, we establish
the short-time existence and uniqueness of smooth solutions to
the combinatorial Yamabe flow.
To address the potential degeneration of triangles along the evolution, we introduce an extended flow with generalized curvature, and establish the global existence of solutions to the extended flow.
Furthermore, under an integrability condition, we establish the uniqueness of solutions to this extended flow, which follows from the stability property of the solutions.
These results provide a local well-posedness theory
for the hyperbolic combinatorial Yamabe flow on
infinitely triangulated surfaces.

\end{abstract}

\maketitle

\section{Introduction}

\subsection{Background}

To study the discrete conformal geometry of polyhedral metrics on manifolds, vertex scaling for Euclidean polyhedral metrics (piecewise linear metrics or PL metrics for short) on triangulated surfaces was independently proposed physically by R\v{o}cek-Williams \cite{Rocek} and mathematically by Luo \cite{Luo1}.
This transformation acts as a natural discrete analogue of the conformal transformations in Riemannian geometry.
Luo \cite{Luo1} also defined the combinatorial Yamabe flow and studied its properties.
Motivated by this framework of vertex scaling for PL metrics on polyhedral surfaces, Bobenko-Pinkall-Springborn \cite{BPS} subsequently proposed the corresponding definition of vertex scaling for hyperbolic polyhedral metrics (piecewise hyperbolic metrics or PH metrics for short ).
They also prove the global rigidity of Euclidean and hyperbolic vertex scalings.
Gu-Luo-Sun-Wu \cite{GLSW} and Gu-Guo-Luo-Sun-Wu \cite{GGLSW} then proved the existence of PL and PH metrics  respectively.
The convergence of vertex scalings has been
extensively studied in the literature.
See \cite{GLW,LSW,Wu,WuZhu} and others.

Finding PL metrics and PH metrics with prescribed combinatorial curvatures on surfaces constitutes a fundamental problem in discrete conformal geometry.
Combinatorial curvature flows serve as a powerful strategy to handle such geometric problems.
A variety of combinatorial curvature flows have been established for vertex scalings of PL metrics and PH metrics, including the combinatorial Yamabe flows
with surgery introduced in \cite{GLSW} and \cite{GGLSW}, where the long-time existence, uniqueness and convergence were proved.
Ge \cite{Gphd} introduced the combinatorial Calabi flow for vertex scaling of PL metrics on triangulated surfaces and proved the short time existence.
The combinatorial Calabi flow for vertex scaling of PH metric with surgery was introduced in \cite{ZXu} and the fractional combinatorial Calabi flow was introduced
in \cite{WuXu}.
For additional investigations into combinatorial curvature flows of vertex scalings for PL and PH metrics, see \cite{GY,Guo, XZ2,ZGZ} and others.

Most existing studies on combinatorial curvature flows have been conducted in the setting of finitely triangulated surfaces, whereas investigations into curvature flows on infinitely triangulated surfaces remain rather scarce.
Recently, Ge-Hua-Zhou \cite{GHZ} adopted parabolic techniques from partial differential equation theory to examine the combinatorial Ricci flow for Thurston’s circle packings on infinite disk triangulated surfaces in Euclidean and hyperbolic background geometries.
Subsequently, Ji \cite{JI} investigated Luo’s combinatorial Yamabe flow \cite{Luo1} for vertex scalings on infinitely triangulated surfaces in Euclidean background geometry.

Motivated by the work in \cite{JI}, we investigate the fundamental properties of the combinatorial Yamabe flow associated with vertex scalings on infinitely triangulated surfaces in hyperbolic background geometry.
In particular, under the assumptions of uniformly bounded
vertex degree and an $\epsilon$-uniformly nondegenerate
initial metric, we establish the short-time existence
and uniqueness of smooth solutions to the hyperbolic
combinatorial Yamabe flow.
A major obstacle in studying the long-time dynamics of the hyperbolic combinatorial Yamabe flow is the potential formation of singularities along the flow evolution.
To overcome this difficulty, we employ an extended notion of combinatorial curvature to treat the possible singularities that may arise along the flow.
With this extension, we establish the long-time existence of solutions to the extended hyperbolic combinatorial Yamabe flow.
Moreover, on infinitely triangulated surfaces, we establish a stability estimate for the extended flow under an integrability condition, from which uniqueness follows as a special case.

Let $M$ be a surface, and $\mathcal{T}=(V,E,F)$ be a triangulation without boundary, where $V$, $E$, and $F$ denote the sets of vertices, edges, and faces respectively.
Here $V$ is an infinite subset of $M$ with $|V|=\infty$.
In this article, we always assume that the triangulation $\mathcal{T}$ is locally finite.
We use $i$, $\{ij\}$, $\triangle ijk$ to denote a vertex, an edge and a face respectively, where $i$, $j$, $k$ are natural numbers.
We denote by $i\sim j$ if the vertex $i$ is adjacent to $j$.
The number of edges incident to vertex $i$ is called its degree and is denoted by $\deg(i)$.
A PH metric on $(M,\mathcal{T})$ is a hyperbolic cone metric on $M$ whose cone singularities are contained in $V$. Equivalently, it can be described by a length function $d: E \to (0, +\infty)$ such that for each face $\triangle ijk \in F$, the three numbers $d_{ij} := d(\{ij\})$, $d_{ik}$, $d_{jk}$ satisfy the strict hyperbolic triangle inequalities, i.e., they can be realized as the side lengths of a non-degenerate hyperbolic triangle. Conversely, given any such length function $d$, one can construct a PH metric by gluing non-degenerate hyperbolic triangles isometrically along their edges in pairs. Thus, we shall henceforth identify a PH metric with its associated length function $d$ on $E$.
The combinatorial curvature $K:V\to (-\infty,2\pi)$ describes the conic singularities of PH metrics at the vertices.
The combinatorial curvature  at each vertex $i$ is
\begin{equation}\label{Eq: K}
	K_i =2\pi- \sum_{\triangle ijk \in F} \theta^{jk}_i,
\end{equation}
where $\theta _{i}^{jk} $ is the interior angle of the hyperbolic triangle $\triangle ijk$ at the vertex $i$.

\begin{definition}\cite{BPS}
	\label{Def: HVS}
Let $d,\tilde{d}:E\to(0,\infty)$ be two PH metrics on $(M,\mathcal{T})$.
We say that $\tilde{d}$ is a hyperbolic vertex scaling of $d$ if there exists a function $u:V\to\mathbb{R}$ such that for every edge $\{ij\}\in E$,
	\begin{equation}
		\sinh\frac{\tilde{d}_{ij}}{2}=e^{\frac{u_i+u_j}{2}}\sinh\frac{d_{ij}}{2}.
	\end{equation}
The function $u$ is called a discrete conformal factor, and we denote the relationship between $d$ and $\tilde{d}$ by $\tilde{d}=u\ast d$.
\end{definition}

\begin{definition}
	\label{Def: CYF}
	Let $(M,\mathcal T)$ be an infinitely triangulated surface, and fix a PH metric $d_0$ on $\mathcal T$.
Given $u^0:V\to\mathbb R$ such that $u^0*d_0$ is a PH
metric, the hyperbolic combinatorial Yamabe flow is
defined by
\begin{equation}\label{Eq: CYF}
	\begin{cases}
		\displaystyle\frac{du_i}{dt}=-K_i(u), & i\in V,\\
		u_i(0)=u_i^0, & i\in V.
	\end{cases}
\end{equation}
Here $K_i(u)$ is the combinatorial curvature of $u*d_0$
at the vertex $i$. The initial metric is $u^0*d_0$.
\end{definition}

\begin{definition}
	Let $d$ be a PH metric on $\mathcal{T}=(V,E,F)$.
	For $\epsilon>0$, we say that $d$ is
	$\epsilon$-uniformly nondegenerate if
	\[
	\min\{\theta_i^{jk},\theta_j^{ik},\theta_k^{ij}\}
	\ge\epsilon
	\]
	for every triangle $\triangle ijk\in F$.
\end{definition}
If $0<\epsilon_1\le\epsilon_2$, then any
$\epsilon_2$-uniformly nondegenerate PH metric is also
$\epsilon_1$-uniformly nondegenerate.

\begin{theorem}[Short-time existence]
	\label{thm :ste}
	Let $(M,\mathcal T)$ be an infinitely triangulated surface
	with vertex degree bounded by $D$, and let $d_0$ be a PH
	metric on $\mathcal T$.
	Suppose that $u^0:V\to\mathbb R$ is such that the initial
	metric $u^0*d_0$ is $\epsilon$-uniformly nondegenerate.
	Then there exists a positive constant $T_0=T_0(\epsilon,D)$
	such that the flow \eqref{Eq: CYF} admits a solution
	$u\in C_t^\infty(V\times[0,T_0])$.
\end{theorem}

Under the same assumptions, we also establish
the short-time uniqueness of solutions.

\begin{theorem}[Short-time uniqueness]
	\label{thm:uq}
	Let $(M,\mathcal T)$ be an infinitely triangulated surface
	with vertex degree bounded by $D$, and let $d_0$ be a PH
	metric on $\mathcal T$.
	Suppose that $u^0:V\to\mathbb R$ is such that the initial
	metric $u^0*d_0$ is $\epsilon$-uniformly nondegenerate.
	Let $u(t)$ and $\hat u(t)$ be two solutions of
	\eqref{Eq: CYF} on $[0,T]$ with
	$u(0)=\hat u(0)=u^0$, where $0<T\le T_0$ and $T_0$ is
	as in Theorem~\ref{thm :ste}.
	Then $u(t)=\hat u(t)$ for all $t\in[0,T]$.
\end{theorem}

The hyperbolic combinatorial Yamabe flow \eqref{Eq: CYF} on infinitely triangulated surfaces may develop singularities.
To handle the potential singularities along this flow,
we continuously extend the discrete curvature $K$ to a generalized curvature $\widetilde{K}$ and then extend the hyperbolic combinatorial Yamabe flow, which remains well-defined when the triangles degenerate.
Consequently, the solution to the extended hyperbolic combinatorial Yamabe flow exists for all time.

\begin{theorem}[Long-time existence of the extended flow]
	\label{thm:longtime}
	Let $(M,\mathcal T)$ be an infinitely triangulated surface,
	and fix a PH metric $d_0$ on $\mathcal T$.
	For every initial value $u^0:V\to\mathbb R$, the extended
	hyperbolic combinatorial Yamabe flow associated with
	$d_0$ admits a global solution
	$u\in C_t^1(V\times[0,\infty))$ with $u(0)=u^0$.
\end{theorem}

The uniqueness of solutions to the extended hyperbolic combinatorial Yamabe flow on finitely triangulated surfaces follows from the convexity argument in \cite{GH, Xu1}.
In this paper, we extend this result to infinitely triangulated surfaces. In fact, we prove a stronger stability property of solutions to the extended flow under an integrability condition in Section 6, from which the uniqueness follows as a special case.

\begin{theorem}[Uniqueness of the extended flow]
	\label{thm:uniqueness}
Let $(M,\mathcal T)$ be an infinitely triangulated surface, and fix a
PH metric $d_0$ on $\mathcal T$.
Let $u(t)$ and $v(t)$ be two global solutions of the
extended hyperbolic combinatorial Yamabe flow associated
with $d_0$, satisfying $u(0)=v(0)$ and
		\begin{equation}  \label{uq:l1}
		\int_0^T
		\sum_{i\in V}\deg(i)|u_i(t)-v_i(t)|\,dt<\infty, \qquad \forall T>0.
	\end{equation}
Then $u(t)\equiv v(t)$ for all $t\ge0$.
\end{theorem}

Theorems~\ref{thm :ste}, \ref{thm:uq}, and
\ref{thm:longtime} provide hyperbolic counterparts of
the corresponding Euclidean results in \cite{JI}.
Our main contributions are threefold.
Firstly, we establish short-time existence under the
assumptions of uniformly bounded vertex degree and an
$\epsilon$-uniformly nondegenerate initial PH metric.
The uniform Delaunay assumption imposed in the
corresponding Euclidean existence theorem in \cite{JI}
is not required.
Secondly, we establish short-time uniqueness under the
same assumptions as those for short-time existence,
without imposing any Delaunay condition on either the
initial metric or the interpolating metrics.
The Euclidean uniqueness proof in \cite{JI} uses a
discrete maximum principle, with uniform Delaunay
assumptions ensuring nonnegative edge weights.
Our proof uses uniform estimates for the partial
derivatives of the combinatorial curvature with respect
to the discrete conformal factors.
Together with the degree bound, these estimates yield
a local Lipschitz bound for the curvature map in the
$L^\infty$ norm, from which uniqueness follows by
Gr\"onwall's inequality.
No sign condition on these partial derivatives is needed.
Thirdly, we establish an $\ell^2$-stability estimate
for pairs of solutions to the extended hyperbolic
combinatorial Yamabe flow under an integrability condition.
This estimate yields the uniqueness result in
Theorem~\ref{thm:uniqueness}.
Neither result is established for the extended Euclidean
flow in \cite{JI}, where uniqueness is explicitly left
open; see \cite[Remark~3.1]{JI}.

The paper is organized as follows. We present preliminary knowledge including the variational principles and curvature evolution equations concerning hyperbolic vertex scalings and the extended hyperbolic combinatorial Yamabe flow in Section 2. In Sections 3 and 4, we establish respectively the short-time existence and uniqueness of solutions to the hyperbolic combinatorial Yamabe flow on infinitely triangulated surfaces.
Section 5 is devoted to the long-time existence of solutions to the extended flow on infinitely triangulated surfaces, while in Section 6 we prove a stability property of solutions to the extended flow on infinitely triangulated surfaces, which in particular implies the uniqueness of solutions.

\section*{Acknowledgements}
The research of the second author is supported by National Natural Science Foundation
of China under grant no. 12471057 and the Fundamental Research Funds for the Central Uni-
versities under grant no. 2042020kf0199.

\section{preliminaries}
\subsection{Variational principles}
To derive the curvature evolution along the hyperbolic combinatorial Yamabe flow, we now recall the variational principle for interior angles under hyperbolic vertex scaling.

\begin{lemma}\label{Lem: VP}
	\cite{GY,XZ,XZ2}
	Let $\triangle ijk$ be a hyperbolic triangle with a PH metric $d_0$, and let $u$ be a discrete conformal factor for the triangle. Denote $d_{st}:=(u*d_0)_{st}$ as the edge length of $\{st\}$ in $\triangle ijk$ and  $\theta^{jk}_{i},\theta^{ik}_{j},\theta^{ij}_{k}$
	as the inner angles opposite to $\{jk\}$, $\{ik\}$, $\{ij\}$ respectively. Then
	\begin{equation}\label{Eq: VP1}
		\frac{\partial \theta^{jk}_i }{\partial u_j} =\frac{\partial \theta^{ik}_j }{\partial u_i} =\frac{\cosh d_{jk}+\cosh d_{ik}-\cosh d_{ij}-1}{A(\cosh d_{ij}+1)} =\frac{1}{2\cosh^2\frac{d_{ij}}{2}} \tan\frac{\theta^{jk}_i+\theta^{ik}_j-\theta^{ij}_k }{2},
	\end{equation}
	where $A=\sinh d_{ik}\sinh d_{ij}\sin\theta^{jk}_i$.
\end{lemma}
We have the following curvature evolution equation along the hyperbolic combinatorial Yamabe flow.
\begin{proposition}\label{prop: CEQ}
	Let $(M,\mathcal{T})$ be an infinitely triangulated surface with a PH metric $d_0$. Let $u: V\to \mathbb{R}$ be a function, and set $d:=u*d_0$.
	Along the hyperbolic combinatorial Yamabe flow, we have the following formula:
	
	\begin{equation}\label{Eq: CEQ}
		\begin{aligned}
			\frac{dK_i}{dt}
			&= \sum_{j\sim i} \left( \frac{\partial \theta^{jk_1}_i}{\partial u_j}
			+ \frac{\partial \theta^{jk_2}_i}{\partial u_j} \right) (K_j - K_i)
			- \sum_{\triangle ijk\in F} \frac{\partial \mathrm{Area}(\triangle ijk)}{\partial u_i} K_i.
		\end{aligned}
	\end{equation}
Here $\triangle ijk_1,\triangle ijk_2\in F$ denote the two triangles adjacent to the edge $\{ij\}$.
\end{proposition}

  The proof is identical to that in \cite[Proposition 3.2]{BL}, so we omit the details.

\begin{remark}\label{remark: darea}
We recall a fundamental variation formula for the areas of hyperbolic triangles, established by Glickenstein and Thomas \cite[Proposition 9]{GT}:
	\begin{equation}\label{eq:ddarea}
		\frac{\partial \mathrm{Area} (\triangle ijk) }{\partial u_k}
		=\frac{\partial\theta ^{jk}_i}{\partial u_k}\big(\cosh d_{ik}-1\big)
		+\frac{\partial\theta ^{ik}_j}{\partial u_k}\big(\cosh d_{jk}-1\big).
	\end{equation}
	Summing over all triangular faces $\triangle ijk\in F$ and using the variational formula\eqref{Eq: VP1}, we obtain

\begin{equation*}
	\begin{aligned}
		&\sum_{\triangle ijk\in F } \frac{\partial \mathrm{Area} (\triangle ijk) }{\partial u_i}\\
		&\quad=\sum_{\triangle ijk\in F }\left[  \frac{\partial\theta ^{jk}_i}{\partial u_j}\big(\cosh d_{ij}-1\big)
		+\frac{\partial\theta ^{jk}_i}{\partial u_k}\big(\cosh d_{ik}-1\big)  \right]\\
		&\quad=\sum_{j\sim i }\left(\frac{\partial \theta^{jk_1}_i }{\partial u_j}
		+\frac{\partial \theta^{jk_2}_i }{\partial u_j}\right)\big(\cosh d_{ij}-1\big)\\
		&\quad=\sum_{j\sim i }\frac{1}{2\cosh^2\frac{d_{ij}}{2}}
		\left(\tan\frac{\theta^{jk_1}_i+\theta^{ik_1}_j-\theta^{ij}_{k_1}}{2}
		+\tan\frac{\theta^{jk_2}_i+\theta^{ik_2}_j-\theta^{ij}_{k_2} }{2}\right)
		\big(\cosh d_{ij}-1\big)\\
		&\quad=\sum_{j\sim i }\tanh^2\frac{d_{ij}}{2}\left(\tan\frac{\theta^{jk_1}_i+\theta^{ik_1}_j-\theta^{ij}_{k_1}}{2}
		+\tan\frac{\theta^{jk_2}_i+\theta^{ik_2}_j-\theta^{ij}_{k_2} }{2}\right),
	\end{aligned}
\end{equation*}
where $\triangle ijk_1,\triangle ijk_2\in F$ are adjacent triangles sharing a common edge $\{ij\}$.

Combining the curvature evolution equation along the hyperbolic combinatorial Yamabe flow \eqref{Eq: CEQ}, we obtain
\begin{equation}\label{eq: dKdt}
	\begin{aligned}
		\frac{dK_i}{dt}
		&= \sum_{j\sim i}\frac{1}{2\cosh^2\frac{d_{ij}}{2}} \Biggl( \tan\frac{\theta^{jk_1}_i + \theta^{ik_1}_j - \theta^{ij}_{k_1}}{2}
		+ \tan\frac{\theta^{jk_2}_i + \theta^{ik_2}_j - \theta^{ij}_{k_2}}{2} \Biggr) (K_j - K_i) \\
		&\qquad - \sum_{j\sim i}\tanh^2\frac{d_{ij}}{2} \Biggl( \tan\frac{\theta^{jk_1}_i + \theta^{ik_1}_j - \theta^{ij}_{k_1}}{2}
		+ \tan\frac{\theta^{jk_2}_i + \theta^{ik_2}_j - \theta^{ij}_{k_2}}{2} \Biggr) K_i.
	\end{aligned}
\end{equation}
\end{remark}

\begin{remark}\label{remark: VP}
	By the definition of combinatorial curvature and an argument analogous to Proposition \ref{prop: CEQ}, we obtain
 \begin{align*}
\frac{\partial K_i}{\partial u_j}=\frac{\partial K_j}{\partial u_i}=- \left(\frac{\partial \theta^{jk_1}_i }{\partial u_j}+\frac{\partial \theta^{jk_2}_i }{\partial u_j} \right),
\end{align*}
 for two adjacent vertices $i$ and $j$ in a triangulation $\mathcal{T}$ with a PH metric. We also have
	\begin{align*}
	\frac{\partial K_i}{\partial u_i}=&\sum_{j\sim i} \left(\frac{\partial \theta^{jk_1}_i }{\partial u_j}+\frac{\partial \theta^{jk_2}_i }{\partial u_j}\right )+\sum_{\triangle ijk\in F}\frac{\partial \mathrm{Area}(\triangle ijk)}{\partial u_i}
\end{align*}
for adjacent triangles $\triangle ijk_1,\triangle ijk_2\in F$.
\end{remark}

\subsection{Admissible space and extension of the flow.}\label{section: EF}
Along the hyperbolic combinatorial Yamabe flow
\eqref{Eq: CYF}, triangles may become degenerate.
To continue the flow through such degeneracies,
we extend the inner angle functions and hence the
combinatorial curvature continuously.
We first introduce the admissible space of PH metrics and the following extension lemma. See \cite{BPS,XZ} for more details.

Denote the admissible space of PH metrics for a triangle $ \triangle ijk\in F$ as $\Omega_{ijk}^H(d_0)$ , i.e.
\begin{align*}
	\Omega_{ijk}^H(d_0) =\{(u_i,u_j,u_k)\in \mathbb{R}^3|(u*d_0)_{rs}+(u*d_0)_{rt}>(u*d_0)_{ts}, \{r,s,t\}=\{i,j,k\}\}.
\end{align*}

\begin{theorem}\label{thm: admiss}
	\cite{BPS,XZ}
	Given any initial nondegenerate hyperbolic discrete metric $u^0*d_0$ on $\triangle ijk$, the admissible space $\Omega_{ijk}^H(d_0)$ of hyperbolic discrete conformal factors $(u_i, u_j, u_k) \in \mathbb{R}^3$ for the triangle $\triangle ijk$ is
	\[
	\Omega_{ijk}^H(d_0) = \mathbb{R}^3 \setminus \cup_{\alpha\in\Lambda} U_\alpha,
	\]
	where $\Lambda = \{i,j,k\}$. $\cup_{\alpha\in\Lambda} U_\alpha$ is a disjoint union of
	
	\[
	U_i = \biggl\{ (u_i,u_j,u_k)\in\mathbb R^3 \,\bigg|\,\xi_i \ge \frac{-B_i+\sqrt{\Delta_i}}{2A_i} \biggr\}
	\]
	with
	\begin{align*}
		A_i &= S_i^4 > 0, \\
		B_i &= -2 S_i^2 (S_j^2\xi_j+S_k^2\xi_k)<0,\\
		\Delta _i&=16S^4_iS^2_jS^2_k\xi_j\xi_k+16S^6_iS^2_jS^2_k>0,
	\end{align*}
	where $S_i=\sinh\frac{(d_{0})_{jk}}{2}$, $\xi_i=e^{-u_i}$.
	As a result, $\Omega_{ijk}^H(d_0)$ is a non-empty and simply connected subset of $\mathbb{R}^3$ with analytic boundary.
\end{theorem}

\begin{lemma}\label{lemma:EXT}\cite{BPS,XZ}
	The inner angles $\theta^{jk}_i, \theta^{ik}_j, \theta^{ij}_k$ defined for $(u_i, u_j, u_k) \in \Omega_{ijk}^H(d_0)$ could be extended to be continuous functions $\widetilde{\theta}^{jk}_i, \widetilde{\theta}^{ik}_j, \widetilde{\theta}^{ij}_k$ defined on $\mathbb{R}^3$, by setting
	\[
	\widetilde{\theta}^{jk}_i(u_i, u_j, u_k) =
	\begin{cases}
		\theta^{jk}_i, & \text{if } (u_i, u_j, u_k) \in \Omega_{ijk}^H, \\
		\pi, & \text{if } (u_i, u_j, u_k) \in U_i, \\
		0, & \text{if } (u_i, u_j, u_k) \in U_j \text{ or } U_k.
	\end{cases}
	\]
\end{lemma}

\begin{remark}\cite{XZ}\label{remark: VP2}
Let $u: V\to \mathbb{R}$ be a function, and set $d:=u* d_0$. If $(u_i, u_j, u_k) \in \Omega_{ijk}^H(d_0)$ tends to a point $(\overline{u}_i,\overline{u}_j,\overline{u}_k) \in \partial U_i$,
we have $\frac{\partial \theta^{jk}_i}{\partial u_j} \to +\infty$, $\frac{\partial \theta^{jk}_i}{\partial u_k} \to +\infty$.
We restate the formula \eqref{eq:ddarea} from Remark \ref{remark: darea}:
\[
\frac{\partial \mathrm{Area}(\triangle ijk)}{\partial u_i} = \frac{\partial \theta^{ik}_j}{\partial u_i} (\cosh d_{ij} - 1) + \frac{\partial \theta^{ij}_k}{\partial u_i} (\cosh d_{ik} - 1).
\]
Then $\frac{\partial \mathrm{Area}(\triangle ijk)}{\partial u_i} \to +\infty$, which implies
\[
\frac{\partial \theta^{jk}_i}{\partial u_i} = -\frac{\partial \mathrm{Area}(\triangle ijk)}{\partial u_i} - \frac{\partial \theta^{ik}_j}{\partial u_i} - \frac{\partial \theta^{ij}_k}{\partial u_i} \to -\infty
\]
as $(u_i,u_j,u_k) \to (\overline{u}_i,\overline{u}_j,\overline{u}_k) \in \partial U_i$. This implies that the extended function $\widetilde{\theta}^{jk}_i$ is not a Lipschitz function.
\end{remark}

A metric $d$ on $(M,\mathcal{T})$ is called a generalized PH metric if $d_{rs}\ge d_{rt}+d_{st}$ for some face $\{r,s,t\}$.
Based on the extension technique in Lemma \ref{lemma:EXT}, the combinatorial curvature $K$, defined on nondegenerate hyperbolic discrete conformal factors in $\Omega^H$, can be continuously extended to the generalized combinatorial curvature. We define
\[
\widetilde{K}_i=2\pi-\sum_{\triangle{ijk} \in F}\tilde{\theta }^{jk}_i,
\]
for any $u: V\to \mathbb{R}$. We call such $u: V\to \mathbb{R}$ a generalized hyperbolic discrete conformal factor.

For any initial value $u^0:V\to\mathbb R$, the extended
hyperbolic combinatorial Yamabe flow is defined by
\begin{equation}\label{Eq: ECYF}
\begin{cases}
	\displaystyle\frac{du_i}{dt}=-\widetilde K_i(u), & i\in V,\\
	u_i(0)=u_i^0, & i\in V.
\end{cases}
\end{equation}
Here $\widetilde K_i(u)$ is the generalized combinatorial
curvature associated with $u*d_0$ at the vertex $i$.

\section{Short-Time Existence of the Hyperbolic Combinatorial Yamabe Flow}

We begin with a basic geometric lemma that provides crucial support for the subsequent analysis of the existence time for the flow \eqref{Eq: CYF}.

\begin{lemma}\label{lem: zl}
	For any hyperbolic triangle $\triangle ijk$, we denote by $\theta_i$, $\theta_j$ and $\theta_k$ the interior angles of the triangle, and by $d_{ij}$, $d_{ik}$ and $d_{jk}$ the edge lengths. Assume
	that $\theta_i,\theta_j,\theta_k \ge \epsilon>0$. Then there exists a constant $\delta=\delta(\epsilon)$ such that for every function $u: V \to \mathbb{R}$ satisfying $\| u \| _{L^{\infty }}\le \delta$, the interior angles $\hat{\theta}_{i}$, $\hat{\theta}_{j}$ and $\hat{\theta}_{k}$ under the conformal metric $\tilde{d}=u*d$ satisfy $|\hat{\theta}_{r}-\theta_{r}|\le \frac{\epsilon}{8}$ for $r \in \{i, j, k\}$.
\end{lemma}
This lemma is elementary in nature, and its proof involves lengthy computations. For brevity, we put the proof in the appendix.

For the proofs in Sections 3 and 4, we may normalize
the initial conformal factor to be zero.
Indeed, set $\bar d_0=u^0*d_0$ and $z=u-u^0$.
By the definition of hyperbolic vertex scaling,
\[
z*\bar d_0=(z+u^0)*d_0=u*d_0.
\]
Thus $z$ satisfies the hyperbolic combinatorial Yamabe
flow with respect to $\bar d_0$, with $z(0)=0$.
Throughout Sections 3 and 4, we use this normalization
and denote $z$ and $\bar d_0$ again by $u$ and $d_0$,
respectively.

We recall several fundamental graph-theoretic notions.

The combinatorial distance $d(i,j)$ between two vertices $i,j\in V$ is defined via
\begin{equation*}
	d(i,j)=\inf\bigl\{k\in\mathbb{N}\,\big|\,\exists\text{ a path }i_0i_1\cdots i_k \text{ with }i_0=i,\ i_k=j\bigr\}.
\end{equation*}
Given a vertex $i\in V$ and an integer $n\ge 0$, the closed combinatorial ball $B_n(i)\subset V$ of radius $n$ centered at $i$ is the set
\begin{equation*}
	B_n(i)=\bigl\{j\in V\mid d(i,j)\le n\bigr\}.
\end{equation*}

Let $(M,\mathcal{T})$ be a locally finite infinitely triangulated surface equipped with a fixed initial PH metric $d_0$.
Let $\mathcal{T}_n$ be the subcomplex of $\mathcal{T}$ consisting of triangles whose vertices are contained in $B_n(i)$.
We denote by $V_n$, $E_n$, and $F_n$ the vertex,
edge, and face sets of $\mathcal T_n$, respectively.
Let $\operatorname{int}(V_n)$ consist of those vertices
in $V_n$ whose incident triangles all belong to $F_n$,
and set
$\partial V_n=V_n\setminus\operatorname{int}(V_n)$.
Then
\[
V_n\subseteq V_{n+1},
\qquad
\bigcup_{n=1}^{\infty}V_n=V.
\]
Moreover, every finite subset of $V$ is contained in
$\operatorname{int}(V_n)$ for all sufficiently large $n$.
We next consider the hyperbolic Yamabe flow on $V_n$.

\begin{align*}
	\begin{cases}
		\displaystyle \frac{du_j^{[n]}}{dt}
		=-K_j(u^{[n]}),
		& j\in\operatorname{int}(V_n),\quad t>0,\\[6pt]
		u_j^{[n]}(0)=0,
		& j\in V_n,\\[6pt]
		u_j^{[n]}(t)=0,
		& j\in\partial V_n,\quad t\ge0.
	\end{cases}
\end{align*}

Since the vertex set $V_n$ is finite, the short-time existence of solutions to the above system follows directly from standard ODE theory.

\textbf{Proof of Theorem \ref{thm :ste}.} Applying Lemma \ref{lem: zl}, there exists a constant $\delta_0 = \delta_0(\epsilon)$ ensuring that the metric $d=u\ast d_0$ satisfies the triangle inequality whenever $u: V \to \mathbb{R}$ satisfies $\|u\|_{L^\infty} \le \delta_0$. In addition, for each triangle $\triangle ijk \in F$, we have
\begin{equation}\label{Eq: Ng1}
	\theta^{jk}_i (d) \ge \frac{7\epsilon}{8}.
\end{equation}
Fix a vertex $j \in V$, and choose a sufficiently large $n$ such that $j \in \operatorname{int}(V_n)$. By the definition of $K_j$, we have
\begin{equation}\label{eq:UK}
	\left|\dfrac{d u_j^{[n]}}{d t}\right|=|K_j(u^{[n]})| \le (2 + \deg(j))\pi \le (2 + D)\pi.
\end{equation}
Therefore, set
\[
T_0 = \frac{\delta_0}{(2+D)\pi},
\]

which depends only on $\epsilon$ and $D$.
As long as the solution exists and $0\le t\le T_0$,
the curvature bound and the initial condition give
\begin{equation}\label{eq :UU}
	\begin{aligned}
		|u_j^{[n]}(t)|
		&=\left|\int_0^t
		\frac{du_j^{[n]}}{ds}(s)\,ds\right|\le \int_0^t |K_j(u^{[n]}(s))|\,ds\le (2+D)\pi t\le\delta_0.
	\end{aligned}
\end{equation}
Together with the zero boundary values, this implies
$\|u^{[n]}(t)\|_{L^\infty(V_n)}\le\delta_0$.
By Lemma~\ref{lem: zl}, each finite-dimensional solution
therefore remains in a compact subset of its admissible
domain for $t\le T_0$.
The standard continuation theorem for ODEs then ensures
that the solution exists on $[0,T_0]$.
Consequently, the angle estimate \eqref{Eq: Ng1} holds
throughout this interval.
From \eqref{eq:UK} and \eqref{eq :UU}, we obtain
\begin{equation}\label{Eq: UC1}
\big\|u^{[n]}_j(t)\big\|_{C^1[0,T_0]}=\sup_{t\in[0,T_0]}|u^{[n]}_j(t)|+\sup_{t\in[0,T_0]}\left|\dfrac{d u_j^{[n]}}{d t}\right|
\le (2 + D)\pi( T_0+1) .
\end{equation}
We proceed to estimate $\displaystyle\big\|u^{[n]}_j(t)\big\|_{C^2[0,T_0]}$.
For the second-derivative estimate, take $n$ sufficiently large that $j$ and all its adjacent vertices belong to $\operatorname{int}(V_n)$.
Combining the curvature evolution equation \eqref{eq: dKdt} from Remark \ref{remark: darea}, we obtain
\begin{equation*}\label{ieq:2Ut}
	\begin{aligned}
	\left|\dfrac{d^2 u_j^{[n]}}{d t^2}\right|=	\left| \frac{dK_j}{dt} \right|
			&\le \sum_{i\sim j}\frac{1}{2\cosh^2\frac{d_{ij}}{2}} \Biggl| \tan\frac{\theta^{jk_1}_i + \theta^{ik_1}_j - \theta^{ij}_{k_1}}{2}
			+ \tan\frac{\theta^{jk_2}_i + \theta^{ik_2}_j - \theta^{ij}_{k_2}}{2} \Biggr| \big|K_j - K_i\big| \\
			&\qquad + \sum_{i\sim j}\tanh^2\frac{d_{ij}}{2} \Biggl| \tan\frac{\theta^{jk_1}_i + \theta^{ik_1}_j - \theta^{ij}_{k_1}}{2}
			+ \tan\frac{\theta^{jk_2}_i + \theta^{ik_2}_j - \theta^{ij}_{k_2}}{2} \Biggr| \big|K_j\big|.
	\end{aligned}
\end{equation*}
By the angle condition for nondegenerate hyperbolic triangles, we have
\begin{equation*}
	\theta^{jk_1}_i+\theta^{ik_1}_j+\theta^{ij}_{k_1} < \pi.
\end{equation*}
Combining this with \eqref{Eq: Ng1} and rearranging terms, we obtain
\begin{equation}\label{eq:uangle}
\frac{7\epsilon}{4}- \frac{\pi}{2}
	< \frac{\theta^{jk_1}_i + \theta^{ik_1}_j - \theta^{ij}_{k_1}}{2}
	< \frac{\pi}{2} - \frac{7\epsilon}{8}.
\end{equation}
Similarly,
\begin{equation*}
\frac{7\epsilon}{4}- \frac{\pi}{2}
	< \frac{\theta^{jk_2}_i + \theta^{ik_2}_j - \theta^{ij}_{k_2}}{2}
	< \frac{\pi}{2} - \frac{7\epsilon}{8},
\end{equation*}
and therefore
\begin{equation}\label{eq:2tan}
	-2\cot\frac{7\epsilon}{4}
	< \tan \frac{\theta^{jk_1}_i + \theta^{ik_1}_j - \theta^{ij}_{k_1}}{2}
	+ \tan \frac{\theta^{jk_2}_i + \theta^{ik_2}_j - \theta^{ij}_{k_2}}{2}
	< 2\cot\frac{7\epsilon}{8}.
\end{equation}
Using the trivial estimate $\cosh^2\frac{d_{ij}}{2}> 1$, we have
\begin{equation}\label{eq:cosh}
0	<\frac{1}{2\cosh^2\frac{d_{ij}}{2}}<\frac{1}{2}.
\end{equation}
By the above inequalities \eqref{eq:2tan} and \eqref{eq:cosh}, we conclude that
\begin{equation}\label{eq:ukiuj}
	\begin{aligned}
		\frac{1}{2\cosh^2\frac{d_{ij}}{2}} \Biggl| \tan\frac{\theta^{jk_1}_i + \theta^{ik_1}_j - \theta^{ij}_{k_1}}{2}
		+ \tan\frac{\theta^{jk_2}_i + \theta^{ik_2}_j - \theta^{ij}_{k_2}}{2} \Biggr| <\cot\frac{7\epsilon}{8}.
	\end{aligned}
\end{equation}
Similarly, in view of the elementary bound $0<\tanh^2\frac{d_{ij}}{2}<1$ and using estimation \eqref{eq:2tan}, we further conclude that
\begin{equation*}
\tanh^2\frac{d_{ij}}{2}	\left|\tan\frac{\theta^{jk_1}_i+\theta^{ik_1}_j-\theta^{ij}_{k_1}}{2}
	+\tan\frac{\theta^{jk_2}_i+\theta^{ik_2}_j-\theta^{ij}_{k_2} }{2} \right|
	< 2\cot\frac{7\epsilon}{8}.
\end{equation*}
Since the vertex degree is bounded by $D$, and by the definition of combinatorial curvature, we have $|K_j - K_i| < \pi D$. Combining with the estimate $|K_i| \le (2+D)\pi$ in \eqref{eq:UK}, we obtain
\begin{equation}\label{ieq:2Ut2}
	\left|\dfrac{d^2u_j^{[n]}}{dt^2}\right|
	< C_0(\epsilon,D), \quad C_0(\epsilon,D) := \pi D(4+3D)\cot\frac{7\epsilon}{8},
\end{equation}
where $C_0(\epsilon,D)$ is a constant depending only on $\epsilon$ and $D$.
We use the constant $T_0$ defined above.
Combining inequalities \eqref{Eq: UC1} and \eqref{ieq:2Ut2}, we obtain
\begin{equation}\label{Eq: UC2}
\big\|u^{[n]}_j(t)\big\|_{C^2[0,T_0]}
= \sup_{t\in[0,T_0]}|u^{[n]}_j(t)| + \sup_{t\in[0,T_0]}\left|\dfrac{d u_j^{[n]}}{d t}\right|
+ \sup_{t\in[0,T_0]}\left|\dfrac{d^2 u_j^{[n]}}{d t^2}\right|
< C(\epsilon,D),
\end{equation}
where $C(\epsilon,D) := (2+D)\pi(T_0+1) + C_0(\epsilon,D)$.
By the above estimations \eqref{Eq: UC1} and \eqref{Eq: UC2}, for each vertex $j\in V$,  the sequence $\{u^{[n]}_j(t)\}_{n=1}^{\infty}$ is uniformly bounded in $C^1$ and $C^2$ on $[0,T_0]$. In particular, the uniform $C^2$-boundedness implies equicontinuity in $C^1[0,T_0]$. Therefore, we can apply the Arzel\`a--Ascoli theorem and the standard diagonal argument.
Hence there exists a subsequence $\{u^{[n_l]}_j(t)\}_{l=1}^{\infty}$ of $\{u^{[n]}_j(t)\}_{n=1}^{\infty}$ such that $u^{[n_l]}_j(t)$ converges uniformly in $C^1[0,T_0]$ to $u^*_j(t)$ as $l\to\infty$. 
Moreover, the limit function $u^*$ satisfies
\[
\begin{cases}
	\displaystyle\frac{du_j^*}{dt}=-K_j(u^*),\\[6pt]
	u_j^*(0)=0,
\end{cases}
\qquad j\in V,\quad t\in[0,T_0].
\]
Since each $K_j$ is a smooth function of finitely many
conformal factors on the admissible domain, it follows
that $u^*\in C_t^\infty(V\times[0,T_0])$.
Returning to the original variables gives the required
solution with initial value $u^0$.
\hfill $\square$

\section{Uniqueness of the Hyperbolic Combinatorial Yamabe Flow}

In this section, we prove the short-time uniqueness of the
hyperbolic combinatorial Yamabe flow. The deduction of
uniqueness from a uniform Lipschitz estimate by
Gr\"onwall's inequality is standard; see also \cite{HLSL}
for a recent application to the Ollivier Ricci flow on
infinite graphs. The main point in our setting is to derive
a uniform Lipschitz estimate for the  combinatorial  curvature
from the hyperbolic angle variation formulas, the uniform
nondegeneracy of the interpolating metrics, and the bounded
vertex degree.

We recall the following form of Gr\"onwall's inequality.
If a nonnegative continuous function $f$ on $[0,T]$ satisfies
\[
f(t)\le C\int_0^t f(\tau)\,d\tau,
\qquad t\in[0,T],
\]
for some constant $C\ge0$, then $f(t)=0$ for all
$t\in[0,T]$.

\textbf{Proof of Theorem~\ref{thm:uq}.}
We use the normalization described in Section 3.
Let $u(t)$ and $\hat u(t)$ be two solutions of the flow
\eqref{Eq: CYF} on $[0,T]$, where $0<T\le T_0$, with
$u(0)=\hat u(0)=0$.
Let $\delta_0$ and $T_0=\delta_0/((2+D)\pi)$ be as in
the proof of Theorem~\ref{thm :ste}.
The curvature bound $|K_i|\le(2+D)\pi$ gives
\[
\begin{aligned}
	\|u(t)\|_{L^\infty}
	&\le (2+D)\pi t\le\delta_0,\\
	\|\hat u(t)\|_{L^\infty}
	&\le (2+D)\pi t\le\delta_0,
\end{aligned}
\qquad t\in[0,T].
\]
For $s\in[0,1]$, set
\[
u_s(t):=su(t)+(1-s)\hat u(t).
\]
Then $\|u_s(t)\|_{L^\infty}\le\delta_0$.
By Lemma~\ref{lem: zl}, the metric $u_s(t)*d_0$ is
$\frac{7\epsilon}{8}$-uniformly nondegenerate
for all $s\in[0,1]$ and $t\in[0,T]$.
We next estimate the derivatives of the combinatorial
curvature at $u_s(t)$.
In the following estimates, all edge lengths and angles
correspond to $d:=u_s(t)*d_0$, and all partial derivatives
are evaluated at $u_s(t)$.
For $j\sim i$, let $\triangle ijk_1$ and $\triangle ijk_2$
be the two triangles adjacent to the edge $\{ij\}$.
By Remark~\ref{remark: VP} and the estimates \ref{eq:ukiuj} in the proof
of Theorem~\ref{thm :ste}, we obtain
\begin{equation}\label{eq:off-diagonal-bound}
	\begin{aligned}
		\left|\frac{\partial K_i}{\partial u_j}\right|
	=
		\left|
		\frac{\partial\theta_i^{jk_1}}{\partial u_j}
		+
		\frac{\partial\theta_i^{jk_2}}{\partial u_j}
		\right|                                         
		\le
			\frac{1}{2\cosh^2(d_{ij}/2)}
		\left|\sum_{\ell=1}^{2}
		\tan
		\frac{
			\theta_i^{jk_\ell}
			+\theta_j^{ik_\ell}
			-\theta_{k_\ell}^{ij}
		}{2}
		\right|                                       \le \cot\frac{7\epsilon}{8}.
	\end{aligned}
\end{equation}
For the partial derivative with respect to $u_i$,
Remarks~\ref{remark: darea} and~\ref{remark: VP} yield

\[
\begin{aligned}
	\frac{\partial K_i}{\partial u_i}
	&=
	\sum_{j\sim i}
	\left(
	\frac{\partial\theta_i^{jk_1}}{\partial u_j}
	+
	\frac{\partial\theta_i^{jk_2}}{\partial u_j}
	\right)
	+
	\sum_{\triangle ijk\in F}
	\frac{\partial\operatorname{Area}(\triangle ijk)}
	{\partial u_i}                                                \\
	&=
	\sum_{j\sim i}
	\left(
	\frac{\partial\theta_i^{jk_1}}{\partial u_j}
	+
	\frac{\partial\theta_i^{jk_2}}{\partial u_j}
	\right)+
	\sum_{j\sim i}
	(\cosh d_{ij}-1)
	\left(
	\frac{\partial\theta_i^{jk_1}}{\partial u_j}
	+
	\frac{\partial\theta_i^{jk_2}}{\partial u_j}
	\right)                                                        \\
	&=
	\sum_{j\sim i}
	\left(
	\frac{\partial\theta_i^{jk_1}}{\partial u_j}
	+
	\frac{\partial\theta_i^{jk_2}}{\partial u_j}
	\right)\cosh d_{ij}.
\end{aligned}
\]
Moreover, by the variational formula \eqref{Eq: VP1} and the angle bound \eqref{eq:uangle},
\[
\begin{aligned}
	\left|
	\cosh d_{ij}
	\frac{\partial\theta_i^{jk}}{\partial u_j}
	\right|
	&=
	\frac{\cosh d_{ij}}{\cosh d_{ij}+1}
	\left|
	\tan\frac{\theta_i^{jk}+\theta_j^{ik}-\theta_k^{ij}}{2}
	\right| \le \cot\frac{7\epsilon}{8}.
\end{aligned}
\]
Since $\deg(i)\le D$, it follows that
\begin{equation}\label{eq:diagonal-bound}
	\left|\frac{\partial K_i}{\partial u_i}\right|
	\le 2D\cot\frac{7\epsilon}{8}.
\end{equation}
Also, $\partial K_i/\partial u_j=0$ whenever
$j\neq i$ and $j\not\sim i$.
Combining \eqref{eq:off-diagonal-bound} and
\eqref{eq:diagonal-bound}, we conclude that
\begin{equation}\label{eq:uuukikj}
\sum_{j\in V}
\left|\frac{\partial K_i}{\partial u_j}(u_s(t))\right|
\le 3D\cot\frac{7\epsilon}{8}
\end{equation}
uniformly for $i\in V$, $s\in[0,1]$, and $t\in[0,T]$.

Set $w(t)=u(t)-\hat u(t)$.
The uniform curvature bound and the flow equations imply
\[
\|w(t)-w(r)\|_{L^\infty}
\le 2(2+D)\pi|t-r|,
\qquad r,t\in[0,T].
\]
In particular, $t\mapsto\|w(t)\|_{L^\infty}$ is continuous
on $[0,T]$.
For each vertex $i\in V$, the combinatorial curvature $K_i$
depends only on the discrete conformal factors at $i$
and its adjacent vertices.
By local finiteness, this dependence involves only finitely
many variables.
Since the interpolating PH metrics $u_s(t)*d_0$ are
$\frac{7\epsilon}{8}$-uniformly nondegenerate, the chain rule and the fundamental theorem
of calculus give
\[
\begin{aligned}
	K_i(u(t))-K_i(\hat u(t))
	&=
	\int_0^1\frac{d}{ds}K_i(u_s(t))\,ds=
	\int_0^1\sum_{j\in V}
	\frac{\partial K_i}{\partial u_j}(u_s(t))w_j(t)\,ds.
\end{aligned}
\]
Using the preceding uniform estimate \eqref{eq:uuukikj} for the curvature
derivatives, we obtain
\begin{equation}\label{eq:chazhiK}
\begin{aligned}
	|K_i(u(t))-K_i(\hat u(t))|
	&\le
	\int_0^1\sum_{j\in V}
	\left|
	\frac{\partial K_i}{\partial u_j}(u_s(t))
	\right|
	|w_j(t)|\,ds \le
	3D\cot\frac{7\epsilon}{8}\,
	\|w(t)\|_{L^\infty}.
\end{aligned}
\end{equation}
Since $u$ and $\hat u$ satisfy the same flow with the same
initial value,
\[
\frac{dw_i}{dt}
=
-\bigl(K_i(u(t))-K_i(\hat u(t))\bigr),
\qquad w_i(0)=0.
\]
Integrating in time and applying  the estimate \eqref{eq:chazhiK},
we obtain
\[
\begin{aligned}
	|w_i(t)|
	&\le
	\int_0^t
	|K_i(u(\tau))-K_i(\hat u(\tau))|\,d\tau \le
	3D\cot\frac{7\epsilon}{8}
	\int_0^t\|w(\tau)\|_{L^\infty}\,d\tau.
\end{aligned}
\]
Since this bound is uniform in $i$, taking the supremum
over $i\in V$ yields
\[
\|w(t)\|_{L^\infty}
\le
3D\cot\frac{7\epsilon}{8}
\int_0^t\|w(\tau)\|_{L^\infty}\,d\tau.
\]
Applying the above form of Gr\"onwall's inequality to the
nonnegative continuous function
$t\mapsto\|w(t)\|_{L^\infty}$, we conclude that
\[
\|w(t)\|_{L^\infty}=0,
\qquad t\in[0,T].
\]
Therefore $u(t)=\hat u(t)$ on $[0,T]$,
which proves the theorem.
\hfill $\square$

	\section{Long-Time Existence of the Extended Hyperbolic Combinatorial Yamabe Flow}
	On a finitely triangulated surface, the extended
	curvatures satisfy
	$|\widetilde K_i|\le (2+\deg(i))\pi$.
	Together with the continuity of the extended curvatures,
	these bounds imply global existence of solutions to the
	extended hyperbolic combinatorial Yamabe flow by standard
	ODE theory \cite{Pontryagin}.
We first establish the long-time existence of solutions to the extended flow \eqref{Eq: ECYF} on an infinitely triangulated surface with hyperbolic background geometry. Our proof follows the same strategy as in \cite{JI}, which established the corresponding result in the Euclidean setting, with suitable modifications adapted to the hyperbolic framework. To this end, we first introduce the following lemma.

\begin{lemma}\label{lemma: evans}
	Let $f : \mathbb{R}^n \to \mathbb{R}$ be a Lipschitz function. Assume that $g$ is a continuous function and $\nabla f = g$ in the weak sense. Then $f \in C^1$ and $\nabla f = g$ everywhere.
\end{lemma}

The proof of this lemma can be found in \cite[Chapter 5, Theorems 5--6]{Evans}.

\textbf{Proof of Theorem \ref{thm:longtime}.}
Consider the extended hyperbolic combinatorial Yamabe flow on $V_n$ given by

\begin{align}\label{eq: ieyf}
\begin{cases}
	\displaystyle\frac{du_j^{[n]}}{dt}
	=-\widetilde K_j(u^{[n]}),
	& j\in\operatorname{int}(V_n),\quad t>0,\\[6pt]
	u_j^{[n]}(0)=u_j^0,
	& j\in V_n,\\[6pt]
	u_j^{[n]}(t)=u_j^0,
	& j\in\partial V_n,\quad t\ge0.
\end{cases}
\end{align}
Fix a vertex $j\in V$, and consider sufficiently large $n$ such that $j\in \operatorname{int}(V_n)$. By the definition of the extended curvature $\widetilde{K}_j$, we have
\begin{equation}\label{eq:ek}
	\left|\frac{du_j^{[n]}}{dt}\right|
	=|\widetilde K_j(u^{[n]})|
	\le (2+\deg(j))\pi.
\end{equation}
For every fixed vertex $j\in V$ and $T>0$, the curvature
bound gives
\[
|u_j^{[n]}(t)-u_j^0|
\le \int_0^t
|\widetilde K_j(u^{[n]}(s))|\,ds
\le (2+\deg(j))\pi t,
\qquad 0\le t\le T,
\]
for all sufficiently large $n$. Consequently,
\[
|u_j^{[n]}(t)|
\le |u_j^0|+(2+\deg(j))\pi T,
\qquad 0\le t\le T.
\]

For each fixed vertex $j\in V$ and each positive integer
$m$, the family $\{u_j^{[n]}\}$ is uniformly bounded and
equicontinuous on $[0,m]$.
Applying the Arzel\`a--Ascoli theorem and a diagonal
argument over the vertices and the intervals $[0,m]$,
we obtain a single subsequence converging uniformly on
every compact time interval at every vertex to a function
$u^*:V\times[0,\infty)\to\mathbb R$.
We relabel this subsequence as $u^{[n]}$.
In particular, $u_j^*(0)=u_j^0$ for every $j\in V$.
Fix $T>0$. It remains to verify that each $u_j^*$ is
$C^1$ on $[0,T]$ and satisfies the evolution equation.

Note that each $u_j^{[n]}$ is Lipschitz continuous with Lipschitz constant bounded by $(2+\deg(j))\pi$. Consequently, the limit $u_j^*$ is also Lipschitz continuous on $[0,T]$. Let $\phi$ be a smooth test function satisfying $\phi(0)=\phi(T)=0$. For each fixed $j\in V$, when $n$ is sufficiently large we have $j \in \operatorname{int}(V_n)$, and hence $\frac{d}{dt}u_j^{[n]} = -\widetilde{K}_j(u^{[n]})$ holds on $[0,T]$. By the dominated convergence theorem,
\[
\int_0^T u_j^* \phi' \,dt
= \lim_{n \to \infty} \int_0^T u_j^{[n]} \phi' \,dt
= \lim_{n \to \infty} \int_0^T \widetilde{K}_j(u^{[n]}) \phi \,dt
= \int_0^T \widetilde{K}_j(u^*) \phi \,dt.
\]
The last equality follows from the continuity of
$\widetilde K_j$, which depends on only finitely many
conformal factors.
Therefore, the identity $\frac{d}{dt}u_j^* = -\widetilde{K}_j(u^*)$ holds in the weak sense. Applying Lemma \ref{lemma: evans}, we conclude $u_j^* \in C^1[0,T]$.  
Since $T>0$ was arbitrary, $u^*$ is a global solution
with initial value $u^0$.
\hfill $\square$

\begin{remark}
By Remark~\ref{remark: VP2}, the derivatives of the
combinatorial curvature with respect to the discrete
conformal factors may become unbounded as triangles
approach degenerate configurations.
The extended curvature map is continuous, but need not
be locally Lipschitz at such configurations.
	 Hence uniqueness is not guaranteed by standard ODE theory. Nevertheless, we prove that under an integrability condition, the extended flow on infinitely triangulated surfaces is unique.
\end{remark}

	\section{Stability and uniqueness of the extended flow}

We first recall some basic properties of PH metrics. For more details on these properties, see \cite{BPS}.
The admissible space of PH metrics for each triangular face $\triangle ijk\in F$ is denoted by $\Omega_{ijk}^H$, and that for the subcomplex $\mathcal{T}_n$ is written as $\Omega^H_{F_n}$ in the following.

Lemma \ref{Lem: VP} shows that the matrix
\begin{equation*}
\Lambda_{ijk}^H = \frac{\partial(\theta_i,\theta_j,\theta_k)}{\partial(u_i,u_j,u_k)}
=
\begin{pmatrix}
	\frac{\partial \theta_i}{\partial u_i} & \frac{\partial \theta_i}{\partial u_j} & \frac{\partial  \theta_i}{\partial u_k}\\[6pt]
	\frac{\partial  \theta_j}{\partial u_i} & \frac{\partial  \theta_j}{\partial u_j} & \frac{\partial  \theta_j}{\partial u_k}\\[6pt]
	\frac{\partial  \theta_k}{\partial u_i} & \frac{\partial  \theta_k}{\partial u_j} & \frac{\partial  \theta_k}{\partial u_k}
\end{pmatrix}
\end{equation*}
is symmetric on $\Omega_{ijk}^H$. Furthermore, one has the following result for the matrix $\Lambda_{ijk}^H$.

\begin{lemma}\label{lem: matrix}\cite{BPS}
		The matrix $\Lambda_{ijk}^H$ is symmetric and negative definite on $\Omega_{ijk}^H$.
\end{lemma}

 Theorem \ref{thm: admiss} and Lemma \ref{lem: matrix} imply the following function
\begin{equation*}
	\mathcal{E}_{ijk}(u_i,u_j,u_k)=\int_{(\overline{u}_i,\overline{u}_j,\overline{u}_k)}^{(u_i,u_j,u_k)} \theta_i du_i+\theta_j du_j+\theta_k du_k
\end{equation*}
is a well-defined, smooth, locally strictly concave function on $\Omega_{ijk}^H$ with $\nabla_{u_i}\mathcal{E}_{ijk}=\theta_i$.
The function $\mathcal{E}_{ijk}$ is called the Ricci energy function for the $\triangle ijk$. Set
\begin{equation*}
	\mathcal{E}_n(u)=2\pi\sum_{i\in V_n} u_i-\sum_{\triangle ijk\in F_n}\mathcal{E}_{ijk}(u_i,u_j,u_k)
\end{equation*}
to be the Ricci energy function defined on the admissible space $\Omega^H_{F_n}$ of nondegenerate hyperbolic discrete conformal factors for $(M,\mathcal{T}_n)$.
 Then $\mathcal{E}_n(u)$ is a locally strictly convex function defined on $\Omega^H_{F_n}$ with
 \begin{equation*}
 \nabla_{u_i}\mathcal{E}_n=2\pi -\sum_{\triangle ijk\in F_n}\theta_i^{jk}.
 \end{equation*}
If $i\in \operatorname{int}(V_n)$, we obtain
 $\nabla_{u_i} \mathcal{E}_n=K_i(u) .$
In the work of Bobenko–Pinkall–Springborn \cite{BPS}, the energy is formulated using the Lobachevsky function. Despite different representations, both functionals possess convexity and induce equivalent gradient flows.

By virtue of the extension results from Bobenko–Pinkall–Springborn in \cite{BPS},
the locally concave function $\mathcal{E}_{ijk}$ for nondegenerate hyperbolic discrete conformal factors on $\triangle ijk$ can be extended to a $C^1$ smooth concave function
\begin{equation*}
	\widetilde{\mathcal{E}}_{ijk}(u_i, u_j, u_k) = \int_{(\overline{u}_i,\overline{u}_j,\overline{u}_k)}^{(u_i, u_j, u_k)} \tilde{\theta}_i \, du_i + \tilde{\theta}_j \, du_j + \tilde{\theta}_k \, du_k
\end{equation*}
defined on $\mathbb{R}^{3} $ with \(\nabla_{u_i} \widetilde{\mathcal{E}}_{ijk} = \tilde{\theta}_i\).
Here $\tilde{\theta}_i$ is the extension of $\theta_i$ given in Lemma \ref{lemma:EXT}.
As a result, the locally convex function $\mathcal{E}_n(u)$  for nondegenerate hyperbolic discrete conformal factors on a triangulated surface $(M, \mathcal{T}_n)$ can be extended to a \(C^1\) smooth convex function

\begin{equation*}
	\widetilde{\mathcal{E}}_n(u) = 2\pi \sum_{i \in V_n} u_i - \sum_{\triangle ijk \in F_n} \widetilde{\mathcal{E}}_{ijk}(u_i, u_j, u_k)
\end{equation*}
defined on $\mathbb{R}^{|V_n|} $ with
\begin{equation*}
\nabla_{u_i} \widetilde{\mathcal{E}}_n  = 2\pi - \sum_{\triangle ijk\in F_n} \tilde{\theta}_i^{jk}.
\end{equation*}
If $i\in \operatorname{int}(V_n)$, we obtain
$ \nabla_{u_i} \widetilde{\mathcal{E}}_n=\widetilde{K}_i(u) .$

Let $p\in[1,\infty)$. We define the $\ell^p$-space over the vertex set $V$ by
\begin{equation*}
	\ell^p(V)
	=
	\biggl\{
	w:V\to\mathbb{R}
	\,\bigg|\,
	\|w\|_{\ell^p(V)}<\infty
	\biggr\},
\end{equation*}
where
\[
\|w\|_{\ell^p(V)}
=
\biggl(\sum_{i\in V}|w_i|^p\biggr)^{1/p}.
\]
The uniqueness result stated in Theorem \ref{thm:uniqueness}
follows as a direct consequence of the following stability estimate.
We therefore prove the stronger stability result below.

\begin{theorem}[Stability]
	\label{thm:s}
	Let $(M,\mathcal T)$ be an infinitely triangulated surface,
	and fix a PH metric $d_0$ on $\mathcal T$.
	Suppose that $u(t)$ and $v(t)$ are two global solutions
	of the extended hyperbolic combinatorial Yamabe flow
	associated with $d_0$, with respective initial values
	$u(0)$ and $v(0)$ satisfying
	\[
	\|u(0)-v(0)\|_{\ell^2(V)}<\infty.
	\]
	Assume that, for every $T>0$,
	\begin{equation}  \label{uq:l1}
		\int_0^T
	\sum_{i\in V}\deg(i)|u_i(t)-v_i(t)|\,dt<\infty.
	\end{equation}
	Then
	\[
	\|u(t)-v(t)\|_{\ell^2(V)}
	\le
	\|u(0)-v(0)\|_{\ell^2(V)}
	\]
	for all $t\ge0$.
\end{theorem}

	To rigorously justify the interchange of infinite summation and time differentiation in the energy estimate, we need a classical result on term-by-term differentiation of absolutely continuous functions \cite[p. 227]{Zhou}.
	
	\begin{lemma}\label{lem:ac-series-diff}\cite{Zhou}
		Let $\{f_i(t)\}_{i\in\mathbb N}\subset AC([0,T))$ be a sequence of absolutely continuous functions on $[0,T)$. Assume that
		\begin{equation*}
				\sum_{i=1}^\infty \int_0^T |f_i'(t)|\,dt<\infty,
		\end{equation*}
	and the series $F(t)=\sum_{i}f_i(t)$ converges at some point in $[0,T)$.
		Then $F(t)$ is absolutely continuous on $[0,T)$, and differentiation commutes with infinite summation almost everywhere:
		\begin{equation*}
			F'(t)=\sum_{i=1}^\infty f_i'(t),\quad \text{a.e. }t\in[0,T).
		\end{equation*}
		\end{lemma}
\textbf{Proof of Theorem \ref{thm:s}.}
Let $u(t)$ and $v(t)$ be the two solutions appearing in
the statement of the theorem, and fix $T>0$.
For each fixed $t\in[0,T)$, write
$u=u(t)$ and $v=v(t)$, and define
\[
\widetilde f_n(s)
=\widetilde{\mathcal E}_n\bigl(su+(1-s)v\bigr),
\qquad s\in[0,1],
\]
where the conformal factors in $\widetilde{\mathcal E}_n$
are restricted to $V_n$.
By the fact that $\widetilde{\mathcal{E}}_n(u)$ is a $C^1$ smooth convex function
defined on $\mathbb{R}^{|V_n|}$, 
it follows that $\widetilde f_n(s)$ is a $C^1$ convex function of $s\in[0,1]$.
A direct computation yields the first-order derivative
\begin{equation*}
	\begin{aligned}
		&\widetilde{f}_n'(s)\\
		&= \nabla\widetilde{\mathcal{E}}_n\bigl(su+(1-s)v\bigr)\cdot(u-v) \\
		&= \sum_{i\in \operatorname{int}(V_n)}\widetilde{K}_i\bigl(su+(1-s)v\bigr)(u_i-v_i)
		+ \sum_{i\in \partial V_n}\biggl(2\pi - \sum_{\triangle ijk\in F_n}
		\tilde{\theta}_i^{jk}\bigl(su+(1-s)v\bigr)\biggr)(u_i-v_i).
	\end{aligned}
\end{equation*}
The convexity of the function $\widetilde{f}_n$ implies that its derivative
$\widetilde{f}_n'(s)$ is nondecreasing on $[0,1]$. Consequently, we obtain the key
monotonicity inequality
\begin{equation}\label{uq:fns}
	\begin{aligned}
		&\widetilde{f}_n'(1)-\widetilde{f}_n'(0)\\
		&=
		\sum_{i\in \operatorname{int}(V_n)}\bigl(\widetilde{K}_i(u)-\widetilde{K}_i(v)\bigr)(u_i-v_i)
		+ \sum_{i\in \partial V_n}\biggl(\sum_{\triangle ijk\in F_n} \tilde{\theta}_i^{jk}\bigl(v\bigr)
		- \sum_{\triangle ijk\in F_n} \tilde{\theta}_i^{jk}\bigl(u\bigr)\biggr)(u_i-v_i)\\
		&\ge 0.
	\end{aligned}
\end{equation}
Since each angle satisfies $|\tilde{\theta}_i^{jk}(v) - \tilde{\theta}_i^{jk}(u)| \le \pi$
, we obtain the estimate
\begin{equation*}
	\begin{aligned}
		\biggl|\sum_{\triangle ijk\in F_n}\tilde{\theta}_i^{jk}\bigl(v\bigr)
		-\sum_{\triangle ijk\in F_n}\tilde{\theta}_i^{jk}\bigl(u\bigr)\biggr|
		\le \deg(i)\pi.
	\end{aligned}
\end{equation*}
We have
\begin{equation*}
	\biggl|
	\sum_{i\in \partial V_n}
	\biggl(
	\sum_{\triangle ijk\in F_n}\tilde{\theta}_i^{jk}\bigl(v\bigr)
	-\sum_{\triangle ijk\in F_n}\tilde{\theta}_i^{jk}\bigl(u\bigr)
	\biggr)(u_i-v_i)
	\biggr|
	\le \pi\sum_{i\in \partial V_n}\deg(i)|u_i-v_i|.
\end{equation*}
Fix an arbitrary \(t\in[0,T)\) at which the weighted series converges:
 \begin{equation*}
 	\sum_{i\in V}\deg(i)|u_i(t)-v_i(t)|<\infty.
 \end{equation*}
By \eqref{uq:l1}, this condition holds for almost every
$t\in[0,T)$.
 In what follows, whenever several statements are said to hold almost everywhere, we implicitly take the union of the corresponding null sets, which is still a null set.
 For any \(\epsilon>0\), the convergence of this weighted series implies that there exists a finite vertex subset \(P\subset V\) such that
 \begin{equation*}
 	\sum_{i\in V\setminus P}\deg(i) |u_i(t)-v_i(t)| < \epsilon.
 \end{equation*}
 Choose $N$ sufficiently large such that
 $P\subset \operatorname{int}(V_N)$. Then, for all $n\ge N$,
 $\partial V_n\subset V\setminus P$, and hence
 \begin{equation*}
 	\sum_{i\in\partial V_n} \deg(i)|u_i(t)-v_i(t)|
 	\le \sum_{i\in V\setminus P} \deg(i)|u_i(t)-v_i(t)|
 	< \epsilon.
 \end{equation*}
 Combining this with the estimate
 \begin{equation*}
 	\biggl|
 	\sum_{i\in \partial V_n}
 	\biggl(
 	\sum_{\triangle ijk\in F_n}\tilde{\theta}_i^{jk}(v)
 	-\sum_{\triangle ijk\in F_n}\tilde{\theta}_i^{jk}(u)
 	\biggr)(u_i-v_i)
 	\biggr|
 	\le \pi\sum_{i\in \partial V_n}\deg(i)|u_i(t)-v_i(t)|,
 \end{equation*}
 we obtain
 \begin{equation*}
 	\biggl|
 	\sum_{i\in \partial V_n}
 	\biggl(
 	\sum_{\triangle ijk\in F_n}\tilde{\theta}_i^{jk}(v)
 	-\sum_{\triangle ijk\in F_n}\tilde{\theta}_i^{jk}(u)
 	\biggr)(u_i-v_i)
 	\biggr|
 	\le \pi\epsilon.
 \end{equation*}
 Taking the limit \(n\to\infty\), we conclude that the boundary term vanishes:
 \begin{equation}\label{uq:blim}
 	\lim_{n\to\infty}
 	\biggl|
 	\sum_{i\in \partial V_n}
 	\biggl(
 	\sum_{\triangle ijk\in F_n}\tilde{\theta}_i^{jk}(v)
 	-\sum_{\triangle ijk\in F_n}\tilde{\theta}_i^{jk}(u)
 	\biggr)(u_i-v_i)
 	\biggr|
 	= 0.
 \end{equation}
 Since the weighted series converges for almost every \(t\in[0,T)\),
 the boundary term vanishes for almost every \(t\in[0,T)\).
Moreover, since $\deg(i)\ge1$, we have $2+\deg(i)\le3\deg(i)$. Hence, by \eqref{uq:l1}, for almost every $t\in[0,T)$, 
 \[
 \sum_{i\in V}
 \left|
 \bigl(\widetilde K_i(u(t))-\widetilde K_i(v(t))\bigr)
 \bigl(u_i(t)-v_i(t)\bigr)
 \right|
 \leq
 2\pi\sum_{i\in V}(2+\deg(i))
 |u_i(t)-v_i(t)|<\infty.
 \]
 Hence the interior sums converge absolutely as $n\to\infty$.
 Therefore, letting \(n\to\infty\) in the finite domain monotonicity inequality,
 we obtain
 \begin{equation}\label{uq: bint}
 	\sum_{i\in V}\bigl(\widetilde{K}_i(u(t))-\widetilde{K}_i(v(t))\bigr)(u_i(t)-v_i(t))
 	\ge 0
 	\qquad \qquad \text{a.e. }t\in[0,T).
 \end{equation}

To finish the proof, we introduce the nonnegative energy functional
\begin{equation*}
	E(t)=\sum_{i\in V}\big(u_i(t)-v_i(t)\big)^2,
\end{equation*}
which satisfies 
$E(t)\ge 0$ for all $t\in [0,T)$.
Since the generalized curvatures satisfy
$|\widetilde{K}_i|\le (2+\deg(i))\pi$,
the functions $u_i(t)$ and $v_i(t)$ possess bounded time derivatives
and are therefore Lipschitz continuous. Hence each map
$t\mapsto \big(u_i(t)-v_i(t)\big)^2$ is absolutely continuous on $[0,T)$,
and its derivative satisfies
\begin{equation*}
	\Bigl|\frac{d}{dt}\bigl(u_i(t)-v_i(t)\bigr)^2\Bigr|
	= 2|u_i(t)-v_i(t)|\cdot\left|\frac{du_i(t)}{dt}-\frac{dv_i(t)}{dt}\right|
	\le 4\pi(2+\deg(i))|u_i(t)-v_i(t)|
\end{equation*}
almost everywhere on $[0,T)$.
By Tonelli's theorem and the condition \eqref{uq:l1}, we have
\begin{equation*}
	\sum_{i\in V}\int_0^T \Bigl|\frac{d}{dt}(u_i-v_i)^2\Bigr|\,dt
	\le 4\pi\int_0^T \sum_{i\in V}(2+\deg(i))|u_i(t)-v_i(t)|\,dt
	< \infty.
\end{equation*}
In addition, the series defining $E(t)$ converges at $t=0$, as $E(0)<\infty$.
All hypotheses of Lemma \ref{lem:ac-series-diff} are satisfied, so summation
and time differentiation commute almost everywhere on $[0,T)$.
Differentiating $E(t)$ yields
\begin{equation*}
	\begin{aligned}
		\frac{dE(t)}{dt}
		&=2\sum_{i\in V}\left(\frac{du_i(t)}{dt}-\frac{dv_i(t)}{dt}\right)\big(u_i(t)-v_i(t)\big)\\
		&=-2\sum_{i\in V}\big(\widetilde{K}_i(u(t))-\widetilde{K}_i(v(t))\big)\big(u_i(t)-v_i(t)\big)\\
		&\le 0, \qquad \text{a.e. }t\in[0,T),
	\end{aligned}
\end{equation*}
where the last inequality follows directly from \eqref{uq: bint}.
As $E(t)$ is absolutely continuous on $[0,T)$, the fundamental theorem of
calculus implies
\[
E(t) - E(0) = \int_0^t \frac{dE}{ds}(s)\,ds
\]
for every $t\in[0,T)$. Since $\frac{dE}{ds}(s)\le 0$ almost everywhere, we obtain
\[
\int_0^t \frac{dE}{ds}(s)\,ds \le 0,
\]
and consequently $E(t)\le E(0)$ on $[0,T)$.
The energy functional $E(t)$ is thus nonnegative and non-increasing along the
flow, leading to the uniform stability estimate
\begin{equation*}
	E(t) = \sum_{i\in V}\big(u_i(t)-v_i(t)\big)^2 \le E(0) 
\end{equation*}
valid for all $t\in[0,T)$. Since \(T>0\) was arbitrary, the stability estimate holds for all \(t\ge0\). In particular, when
$u(0)=v(0)$, we obtain $u(t)=v(t)$ for all $t\geq0$,
which proves Theorem \ref{thm:uniqueness}.
\hfill $\square$

\begin{remark}
	We note that \cite{JI} establishes the long-time existence of the
	extended Euclidean combinatorial Yamabe flow, but does not address
	its stability or uniqueness. The stability estimate established in
	the proof of Theorem \ref{thm:s} also applies to the extended
	Euclidean setting considered in \cite{JI}. In particular, solutions with the same initial value
	are unique under the same integrability condition.
\end{remark}

\section{Appendix}

\textbf{Proof of Lemma 3.1}

By the hyperbolic cosine law expressing $d_{ij}$ in terms of $\theta^{ij}_k$, $\theta^{jk}_i$, $\theta^{ik}_j$, we have

\begin{equation*}
	\begin{aligned}
		\sinh^{2}\frac{d_{ij}}{2} &=\frac{1}{2} (\cosh d_{ij}-1)\\
		&=\frac{1}{2} \left( \frac{\cos\theta^{ij}_k+\cos\theta^{jk}_i \cos\theta^{ik}_j}{\sin\theta^{jk}_i\sin\theta^{ik}_j} -1 \right)\\
		&=\frac{1}{2} \frac{\cos\theta^{ij}_k+\cos(\theta^{jk}_i+\theta^{ik}_j)}{\sin\theta^{jk}_i\sin\theta^{ik}_j}\\
		&=\frac{\cos\frac{\theta^{ij}_k+\theta^{jk}_i+\theta^{ik}_j}{2}\cos\frac{\theta^{ij}_k-\theta^{jk}_i-\theta^{ik}_j}{2}}{\sin\theta^{jk}_i\sin\theta^{ik}_j}.
	\end{aligned}
\end{equation*}
Thus,
\begin{equation*}
		\sinh^2\frac{d_{ij}}{2}<\frac{1}{\sin^2\epsilon}.
\end{equation*}
Similarly, we have
\begin{equation*}
	\sinh^2\frac{d_{ik}}{2}<\frac{1}{\sin^2\epsilon},\quad 	\sinh^2\frac{d_{jk}}{2} <\frac{1}{\sin^2\epsilon}.
\end{equation*}
Since the triangle $\triangle ijk$ is uniformly nondegenerate, we obtain $2\epsilon-\pi<\theta^{ik}_j-\theta^{jk}_i-\theta^{ij}_k<\pi-4\epsilon$. Therefore,
\begin{align*}
\quad \frac{\sinh^2\frac{d_{ij}}{2}}{\sinh^2\frac{d_{ik}}{2}}=\frac{\cos\frac{\theta^{ij}_k-\theta^{jk}_i-\theta^{ik}_j}{2}\sin\theta^{ij}_k}{\cos\frac{\theta^{ik}_j-\theta^{jk}_i-\theta^{ij}_k}{2}\sin\theta^{ik}_j}<\frac{1}{\sin^2\epsilon}.
\end{align*}
Similarly, we have
\begin{align*}
	\quad \frac{\sinh^2\frac{d_{jk}}{2}}{\sinh^2\frac{d_{ij}}{2}}<\frac{1}{\sin^2\epsilon},\quad \frac{\sinh^2\frac{d_{jk}}{2}}{\sinh^2\frac{d_{ik}}{2}}<\frac{1}{\sin^2\epsilon}.
\end{align*}
We first carry out the following estimates whenever
the scaled edge lengths satisfy the strict triangle
inequalities. We justify this condition for the chosen
$\delta$ at the end of the proof.
By the hyperbolic cosine law,
\begin{equation*}
	\begin{aligned}
		&\left|\cos\theta_i - \cos\hat{\theta}_i\right|\\
		&= \left| \frac{\cosh d_{ij} \cosh d_{ik} - \cosh d_{jk}}{\sinh d_{ij} \sinh d_{ik}}
		- \frac{\cosh \tilde{d}_{ij} \cosh \tilde{d}_{ik} - \cosh \tilde{d}_{jk}}{\sinh \tilde{d}_{ij} \sinh \tilde{d}_{ik}} \right| \\[6pt]
		&= \left| \frac{4\sinh^2\frac{d_{ij}}{2}\sinh^2\frac{d_{ik}}{2}
			- 2\sinh^2\frac{d_{jk}}{2} + 2\sinh^2\frac{d_{ij}}{2} + 2\sinh^2\frac{d_{ik}}{2}}{\sinh d_{ij}\sinh d_{ik}} \right. \\[6pt]
		&\quad \left. - \frac{4\sinh^2\frac{\tilde{d}_{ij}}{2}\sinh^2\frac{\tilde{d}_{ik}}{2}
			- 2\sinh^2\frac{\tilde{d}_{jk}}{2} + 2\sinh^2\frac{\tilde{d}_{ij}}{2} + 2\sinh^2\frac{\tilde{d}_{ik}}{2}}{\sinh \tilde{d}_{ij}\sinh \tilde{d}_{ik}} \right|\\
		&\le \left|\frac{2\sinh^2\frac{d_{ij}}{2} }{\sinh d_{ij}\sinh d_{ik}} -\frac{2\sinh^2\frac{\tilde d_{ij}}{2} }{\sinh\tilde d_{ij}\sinh\tilde d_{ik}}\right|+\left|\frac{2\sinh^2\frac{d_{ik}}{2} }{\sinh d_{ij}\sinh d_{ik}} -\frac{2\sinh^2\frac{\tilde d_{ik}}{2} }{\sinh\tilde d_{ij}\sinh\tilde d_{ik}}\right|\\
		&+\left|\frac{2\sinh^2\frac{d_{jk}}{2} }{\sinh d_{ij}\sinh d_{ik}} -\frac{2\sinh^2\frac{\tilde d_{jk}}{2} }{\sinh\tilde d_{ij}\sinh\tilde d_{ik}}\right|+\left|\frac{4\sinh^2\frac{d_{ij}}{2}\sinh^2\frac{d_{ik}}{2} }{\sinh d_{ij}\sinh d_{ik}} -\frac{4\sinh^2\frac{\tilde d_{ij}}{2}\sinh^2\frac{\tilde d_{ik}}{2} }{\sinh\tilde d_{ij}\sinh\tilde d_{ik}}\right|.
			\end{aligned}
\end{equation*}
We can estimate this expression by decomposing it into four terms. Let us first estimate the first term,
\begin{equation*}
	\begin{aligned}
		& \left| \frac{2\sinh^2 \frac{d_{ij}}{2} }{\sinh d_{ij}\sinh d_{ik}}
		- \frac{2\sinh^2 \frac{\tilde d_{ij}}{2} }{\sinh\tilde d_{ij}\sinh\tilde d_{ik}} \right| \\
		=&\, 2\left| \frac{\sinh^2 \frac{d_{ij}}{2}}
		{4\sinh \frac{d_{ij}}{2}\cosh \frac{d_{ij}}{2}\sinh \frac{d_{ik}}{2}\cosh \frac{d_{ik}}{2}}
		- \frac{e^{u_i+u_j}\sinh^2 \frac{d_{ij}}{2}}
		{4e^{u_i+\frac{u_j+u_k}{2}}\sinh \frac{d_{ij}}{2}
			\cosh \frac{\tilde d_{ij}}{2}\sinh \frac{d_{ik}}{2}\cosh \frac{\tilde d_{ik}}{2}} \right| \\
		=&\, \frac{\sinh \frac{d_{ij}}{2}}{2\sinh \frac{d_{ik}}{2}}
		\left| \frac{1}{\cosh \frac{d_{ij}}{2}\cosh \frac{d_{ik}}{2}}
		- \frac{e^{\frac{u_j-u_k}{2}}}{\cosh \frac{\tilde d_{ij}}{2}\cosh \frac{\tilde d_{ik}}{2}} \right| \\
		=&\, \frac{\sinh \frac{d_{ij}}{2}}
		{2\sinh \frac{d_{ik}}{2}\cosh \frac{\tilde d_{ij}}{2}\cosh \frac{\tilde d_{ik}}{2}}
		\left| \frac{\cosh \frac{\tilde d_{ij}}{2}\cosh \frac{\tilde d_{ik}}{2}}
		{\cosh \frac{d_{ij}}{2}\cosh \frac{d_{ik}}{2}}
		- e^{\frac{u_j-u_k}{2}} \right| \\
		\le&\, \frac{1}{2\sin \epsilon}
		\left[ \left| \frac{\cosh \frac{\tilde d_{ij}}{2}\cosh \frac{\tilde d_{ik}}{2}}
		{\cosh \frac{d_{ij}}{2}\cosh \frac{d_{ik}}{2}} - 1 \right|
		+ \left| 1 - e^{\frac{u_j-u_k}{2}} \right| \right].
	\end{aligned}
\end{equation*}
Note that

\begin{equation*}
	\begin{aligned}
		&\left| \frac{\cosh \frac{\tilde d_{ij}}{2}\cosh \frac{\tilde d_{ik}}{2}}{\cosh \frac{d_{ij}}{2}\cosh \frac{d_{ik}}{2}} - 1 \right|\\
		=& \frac{\left| \cosh^2 \frac{\tilde d_{ij}}{2}\cosh^2 \frac{\tilde d_{ik}}{2} - \cosh^2 \frac{d_{ij}}{2}\cosh^2 \frac{d_{ik}}{2} \right|}{\cosh \frac{d_{ij}}{2}\cosh \frac{d_{ik}}{2}\left( \cosh \frac{\tilde d_{ij}}{2}\cosh \frac{\tilde d_{ik}}{2} + \cosh \frac{d_{ij}}{2}\cosh \frac{d_{ik}}{2} \right)} \\
		\le& \left| \cosh^2 \frac{\tilde d_{ij}}{2}\cosh^2 \frac{\tilde d_{ik}}{2} - \cosh^2 \frac{d_{ij}}{2}\cosh^2 \frac{d_{ik}}{2} \right| \\
		=& \left| \left( 1+e^{u_i+u_j}\sinh^2 \frac{d_{ij}}{2} \right)\!
		\left( 1+e^{u_i+u_k}\sinh^2 \frac{d_{ik}}{2} \right)
		- \left( 1+\sinh^2 \frac{d_{ij}}{2} \right)\!
		\left( 1+\sinh^2 \frac{d_{ik}}{2} \right) \right| \\
		\le& \sinh^2 \frac{d_{ij}}{2}\left| e^{u_i+u_j}-1 \right|
		+ \sinh^2 \frac{d_{ik}}{2}\left| e^{u_i+u_k}-1 \right| \\
		&+ \sinh^2 \frac{d_{ij}}{2}\sinh^2 \frac{d_{ik}}{2}\left| e^{2u_i+u_j+u_k}-1 \right| \\
		<& \frac{3}{\sin^4\epsilon}\left( e^{4\delta}-1 \right).
	\end{aligned}
\end{equation*}
Thus
\begin{equation*}
	\begin{aligned}
		\left|\frac{2\sinh^2\frac{d_{ij}}{2} }{\sinh d_{ij}\sinh d_{ik}} -\frac{2\sinh^2\frac{\tilde d_{ij}}{2} }{\sinh\tilde d_{ij}\sinh\tilde d_{ik}}\right|<\frac{2(e^{4\delta}-1)}{\sin^5\epsilon}.
		\end{aligned}
\end{equation*}
Similarly, the second term can be estimated as follows,
\begin{equation*}
	\begin{aligned}
		\left|\frac{2\sinh^2\frac{d_{ik}}{2} }{\sinh d_{ij}\sinh d_{ik}} -\frac{2\sinh^2\frac{\tilde d_{ik}}{2} }{\sinh\tilde d_{ij}\sinh\tilde d_{ik}}\right|<\frac{2(e^{4\delta}-1)}{\sin^5\epsilon}.
	\end{aligned}
\end{equation*}
We now estimate the third term,
\begin{equation*}
	\begin{aligned}
		& \left| \frac{2\sinh^2 \frac{d_{jk}}{2} }{\sinh d_{ij}\sinh d_{ik}}
		- \frac{2\sinh^2 \frac{\tilde d_{jk}}{2} }{\sinh\tilde d_{ij}\sinh\tilde d_{ik}} \right| \\
		=&\, 2\left| \frac{\sinh^2 \frac{d_{jk}}{2}}
		{4\sinh \frac{d_{ij}}{2}\cosh \frac{d_{ij}}{2}\sinh \frac{d_{ik}}{2}\cosh \frac{d_{ik}}{2}}
		- \frac{e^{u_j+u_k}\sinh^2 \frac{d_{jk}}{2}}
		{4e^{u_i+\frac{u_j+u_k}{2}}\sinh \frac{d_{ij}}{2}\cosh \frac{\tilde d_{ij}}{2}
			\sinh \frac{d_{ik}}{2}\cosh \frac{\tilde d_{ik}}{2}} \right| \\
		=&\, \frac{\sinh^2 \frac{d_{jk}}{2}}{2\sinh \frac{d_{ij}}{2}\sinh \frac{d_{ik}}{2}}
		\left| \frac{1}{\cosh \frac{d_{ij}}{2}\cosh \frac{d_{ik}}{2}}
		- \frac{e^{\frac{u_j+u_k}{2}-u_i}}{\cosh \frac{\tilde d_{ij}}{2}\cosh \frac{\tilde d_{ik}}{2}} \right| \\
		=&\, \frac{\sinh^2 \frac{d_{jk}}{2}}
		{2\sinh \frac{d_{ij}}{2}\sinh \frac{d_{ik}}{2}\cosh \frac{\tilde d_{ij}}{2}\cosh \frac{\tilde d_{ik}}{2}}
		\left| \frac{\cosh \frac{\tilde d_{ij}}{2}\cosh \frac{\tilde d_{ik}}{2}}
		{\cosh \frac{d_{ij}}{2}\cosh \frac{d_{ik}}{2}}
		- e^{\frac{u_j+u_k}{2}-u_i} \right| \\
		\le&\, \frac{1}{2\sin^2\epsilon}
		\left[ \left| \frac{\cosh \frac{\tilde d_{ij}}{2}\cosh \frac{\tilde d_{ik}}{2}}
		{\cosh \frac{d_{ij}}{2}\cosh \frac{d_{ik}}{2}} - 1 \right|
		+ \left| 1-e^{\frac{u_j+u_k}{2}-u_i} \right| \right] \\
		<&\, \frac{1}{2\sin^2\epsilon}
		\left( \frac{3\left(e^{4\delta}-1\right)}{\sin^4\epsilon} + e^{4\delta}-1 \right) \\
		<&\, \frac{2(e^{4\delta}-1)}{\sin^6\epsilon}.
	\end{aligned}
\end{equation*}
We now estimate the last term,
\begin{equation*}
	\begin{aligned}
		& \left| \frac{4\sinh^2 \frac{d_{ij}}{2}\sinh^2 \frac{d_{ik}}{2} }{\sinh d_{ij}\sinh d_{ik}}
		- \frac{4\sinh^2 \frac{\tilde d_{ij}}{2}\sinh^2 \frac{\tilde d_{ik}}{2} }{\sinh\tilde d_{ij}\sinh\tilde d_{ik}} \right| \\
		=&\, \left| \frac{\sinh \frac{d_{ij}}{2}\sinh \frac{d_{ik}}{2} }{\cosh \frac{d_{ij}}{2}\cosh \frac{d_{ik}}{2}}
		- \frac{\sinh \frac{\tilde d_{ij}}{2}\sinh \frac{\tilde d_{ik}}{2} }{\cosh \frac{\tilde d_{ij}}{2}\cosh \frac{\tilde d_{ik}}{2}} \right| \\
		=&\, \left| \frac{\sinh \frac{d_{ij}}{2}\sinh \frac{d_{ik}}{2} }{\cosh \frac{d_{ij}}{2}\cosh \frac{d_{ik}}{2}}
		- \frac{e^{u_i+\frac{u_j+u_k}{2}}\sinh \frac{d_{ij}}{2}\sinh \frac{d_{ik}}{2} }{\cosh \frac{\tilde d_{ij}}{2}\cosh \frac{\tilde d_{ik}}{2}} \right| \\
		=&\, \frac{\sinh \frac{d_{ij}}{2}\sinh \frac{d_{ik}}{2}}{\cosh \frac{\tilde d_{ij}}{2}\cosh \frac{\tilde d_{ik}}{2}}
		\left| \frac{\cosh \frac{\tilde d_{ij}}{2}\cosh \frac{\tilde d_{ik}}{2} }{\cosh \frac{d_{ij}}{2}\cosh \frac{d_{ik}}{2}}
		- e^{u_i+\frac{u_j+u_k}{2}} \right| \\
		\le&\, \sinh \frac{d_{ij}}{2}\sinh \frac{d_{ik}}{2}
		\left[ \left| \frac{\cosh \frac{\tilde d_{ij}}{2}\cosh \frac{\tilde d_{ik}}{2}}{\cosh \frac{d_{ij}}{2}\cosh \frac{d_{ik}}{2}} - 1 \right|
		+ \left| 1 - e^{u_i+\frac{u_j+u_k}{2}} \right| \right] \\
		<&\, \frac{1}{\sin^2 \epsilon}
		\left( \frac{3\left(e^{4\delta}-1\right)}{\sin^4\epsilon} + e^{4\delta}-1 \right) \\
		<&\, \frac{4\left(e^{4\delta}-1\right)}{\sin^6\epsilon}.
	\end{aligned}
\end{equation*}
From the above calculations, we conclude that
\begin{align*}
	\left|\cos\theta_i - \cos\hat{\theta}_i\right|<\frac{10(e^{4\delta}-1)}{\sin^6\epsilon}.
\end{align*}
Let
\begin{align*}
	\delta = \frac{1}{4}\log\left(1 + \frac{\sin^6\epsilon}{10}\left(1 - \cos\frac{\epsilon}{8}\right)\right).
\end{align*}
This choice gives
\[
|\cos\theta_r-\cos\hat\theta_r|
\le 1-\cos\frac{\epsilon}{8},
\qquad r\in\{i,j,k\}.
\]
For any $a,b\in[0,\pi]$, setting $h=|a-b|$, we have
\[
|\cos a-\cos b|
=2\sin\frac{a+b}{2}\sin\frac h2
\ge 2\sin^2\frac h2
=1-\cos h.
\]
Indeed, $(a+b)/2\in[h/2,\pi-h/2]$.
It follows that
\[
|\hat\theta_r-\theta_r|\le\frac{\epsilon}{8}
\]
whenever the scaled triangle is nondegenerate.

To verify nondegeneracy, consider the continuous path
$d(s)=(su)*d$, $s\in[0,1]$.
The triangle is nondegenerate for $s$ sufficiently close
to $0$. Since $\|su\|_{L^\infty}\le\delta$, the preceding
estimates apply as long as the triangle is nondegenerate
and give an angle lower bound of $7\epsilon/8$.
The edge lengths remain positive along this path.
At a first degenerate parameter value, continuity of
the extended angle functions would force an angle to
tend to zero, contradicting this lower bound.
Thus $d(s)$ is nondegenerate for every $s\in[0,1]$,
and the desired estimate follows at $s=1$.
\hfill $\square$


\end{document}